\documentclass[12pt]{article}
\usepackage{amssymb,amsmath,amsfonts}
\usepackage{amsthm}
\usepackage{cite}
\usepackage[textwidth=160mm,textheight=220mm]{geometry}

\usepackage{graphicx}
\usepackage[dvips]{epsfig}
\usepackage{subfigure}
\usepackage[active]{srcltx}

\newtheorem{lemm}{Lemma}[section]
\newtheorem{theo}{Theorem}[section]

\newtheorem{prop}{Proposition}[section]

\theoremstyle{definition}
\newtheorem{Defi}{Definition}[section]

\newtheorem{obs}{Remark}[section]

\numberwithin{equation}{section}

\begin{document}

\title{On radial stationary solutions to a model of nonequilibrium growth}

\bigskip

\author{Carlos Escudero\footnote{Supported by projects MTM2010-18128, RYC-2011-09025 and SEV-2011-0087.}
\\ [4mm] {\small Departamento de Matem\'aticas
\& ICMAT (CSIC-UAM-UC3M-UCM),}
\\{\small Universidad Aut\'onoma de Madrid, E-28049 Madrid, Spain}\\
[4mm] Robert Hakl\footnote{Supported by project RUO:67985840.}\\
[4mm] {\small Institute of Mathematics AS CR, \v Zi\v zkova 22,
616 62 Brno, Czech Republic}\\ [4mm]
Ireneo Peral\footnote{Supported by project MTM2010-18128.}\\
[4mm] {\small Departamento de Matem\'aticas,} \\{\small Universidad
Aut\'onoma de Madrid, E-28049 Madrid, Spain} \\ [4mm] Pedro J.
Torres\footnote{Supported by project MTM2011-23652.}\\
[4mm] {\small Departamento de Matem\'atica Aplicada,}
\\ {\small Universidad de Granada, E-18071 Granada, Spain.}
}

\date{}

\maketitle

\begin{abstract}
We present the formal geometric derivation of a nonequilibrium
growth model that takes the form of a parabolic partial
differential equation. Subsequently, we study its stationary
radial solutions by means of variational techniques. Our
results depend on the size of a parameter that plays the role of
the strength of forcing. For small forcing we prove the existence
and multiplicity of solutions to the elliptic problem.
We discuss our results in the context
of nonequilibrium statistical mechanics.
\end{abstract}

\section{Introduction}

Epitaxial growth is characterized by the deposition of new
material on existing layers of the same material under high vacuum
conditions. This technique is used in the semiconductor industry
for the growth of thin films~\cite{barabasi}. The crystals grown
may be composed of a pure chemical element like silicon or
germanium, or may either be an alloy like gallium arsenide or
indium phosphide. In case of molecular beam epitaxy the deposition
takes place at a very slow rate and almost atom by atom. The goal
in most situations of thin film growth is growing an ordered
crystal structure with flat surface. But in epitaxial growth it is
quite usual finding a mounded structure generated along the
surface evolution~\cite{lengel}. The actual origin of this mounded
structure is to a large extent unknown, although some mechanisms
(like energy barriers) have already been proposed. Attempting to
perform \emph{ab initio} quantum mechanical calculations in this
system is computationally too demanding, what opens the way to the
introduction of simplified models. These have been usually
developed within the realm of non-equilibrium statistical
mechanics, and can be of a discrete probabilistic nature or have
the form of a differential equation~\cite{barabasi}. Discrete
models usually represent adatoms (the atoms deposited on the
surfaces) as occupying lattice sites. They are placed randomly at
one such site and then they are allowed to move according to some
rules which characterize the different models. A different
modelling possibility is using partial differential equations,
which in this field are frequently provided with stochastic
forcing terms. In this work we will focus on rigorous and
numerical analyses of ordinary differential equations related to
models which have been introduced in the context of epitaxial
growth. We hope that a systematic mathematical study will
contribute to the understanding of this sort of processes, which
are relevant both in pure physics and its industrial applications,
in the long term.

The mathematical description of epitaxial growth uses the function
\begin{equation}
u: \Omega \subset \mathbb{R}^2 \times \mathbb{R}^+ \rightarrow
\mathbb{R},
\end{equation}
which describes the height of the growing interface in the spatial
point $x \in \Omega \subset \mathbb{R}^2$ at time $t \in
\mathbb{R}^+$. Although this theoretical framework can be extended
to any spatial dimension $N$, we will concentrate here on the
physical situation $N=2$. A basic modelling assumption is of
course that $u$ is an univalued function, a fact that holds in a
reasonably large number of cases~\cite{barabasi}. The macroscopic
description of the growing interface is given by a partial
differential equation for $u$ which is usually postulated using
phenomenological and symmetry arguments~\cite{barabasi,marsili}. A
prominent example of such a theory is given by the
Kardar-Parisi-Zhang equation~\cite{kpz}
\begin{equation}
u_t = \nu \Delta u + \gamma |\nabla u|^2 + \eta(x,t),
\end{equation}
which has been extensively studied in the physical literature and
it is currently being investigated for its interesting
mathematical properties~\cite{ireneo1,ireneo2}. It has been argued
however that epitaxial growth processes should be described by
some equation coming from a conservation law and, in particular, that
the term $|\nabla u|^2$ should not be present in such an equation~\cite{barabasi}. To this end,
among others, the conservative counterpart of the Kardar-Parisi-Zhang
equation was introduced~\cite{sun,villain,vlds}
\begin{equation}
\label{ssg} u_t = -\mu \Delta^2 u + \kappa \Delta |\nabla u|^2 +
\zeta(x,t).
\end{equation}
This equation is conservative in the sense that the first moment
$\int_\Omega u \, dx$ is constant if the appropriate boundary
conditions are used. It can be considered as a higher order
counterpart of the Kardar-Parisi-Zhang equation, and it poses as
well a number of fundamental mathematical
questions~\cite{dirk1,dirk2,dirk3}.

In this work we will focus on a variation of the last equation.
Its formal derivation will be presented in the following section.
The remainder of this work will be devoted to clarify the
analytical properties of the radial stationary solutions to the
model under consideration.

\section{Formal derivation of the model}

Herein we will adopt a variational formulation of the surface
growth equation, which has been postulated as a simple and yet
physically relevant way of developing growth
models~\cite{marsili}. In order to proceed with our formal
derivation, we will assume that the height function obeys a
gradient flow equation with a forcing term
\begin{equation}
u_t= \sqrt{1 + (\nabla u)^2} \left[ - \frac{\delta
\mathcal{J}(u)}{\delta u} + \xi(x,t) \right].
\end{equation}
The functional $\mathcal{J}$ denotes a potential which describes
the microscopic properties of the interface and, at the
macroscopic scale, it is assumed that it can be expressed as a
function of the surface mean curvature only~\cite{marsili}
\begin{equation}
\mathcal{J}(u)= \int_\Omega F(H) \, \sqrt{1 + (\nabla u)^2} \,\,
dx,
\end{equation}
where the presence of the square root terms models growth along
the normal to the surface, $H$ denotes the mean curvature and $F$
is an unknown function of $H$. We will furthermore assume that
this function can be expanded in a power series
\begin{equation}
\label{expansion} F(H)= K_0+K_1 H + \frac{K_2}{2} H^2 +
\frac{K_3}{6} H^3 + \cdots ,
\end{equation}
and subsequently formally apply the small gradient expansion,
which assumes $|\nabla u| \ll 1$. This is a classical
approximation in this physical context~\cite{marsili} and it is
basic in the derivation of the Kardar-Parisi-Zhang
equation~\cite{kpz} among others. In the resulting equation, only
linear and quadratic terms in the field $u$ and its derivatives
are retained, as higher order nonlinearities are assumed not to be
relevant in the large scale description of a growing
interface~\cite{barabasi}. The final result reads
\begin{equation}
\label{parabolic} u_t = K_0 \, \Delta u + 2 \, K_1 \, \det \left(
D^2 u \right) - K_2 \, \Delta^2 u - \frac{1}{2} \, K_3 \, \Delta
\left( \Delta u \right)^2 + \xi(x,t),
\end{equation}
which is, as well as~(\ref{ssg}), a conservative equation in the
sense that $\int_\Omega u \, dx$ is constant if appropriate
boundary conditions are used. We note that powers of the mean
curvature higher than the cubic one in expansion~(\ref{expansion})
do not contribute to equation~(\ref{parabolic}) as they imply
cubic or higher nonlinearities of the field $u$ or its
derivatives. The terms in equation~(\ref{parabolic}) have a clear
geometrical meaning. The term proportional to $K_0$ is the result
of the minimization of the zeroth order of the mean curvature,
that is, it corresponds to the minimization of the surface area.
Its functional form simply reduces to standard diffusion. The term
proportional to $K_1$ comes from the minimization of the mean
curvature and actually it is the determinant of the Hessian
matrix, which is nothing but the small gradient approximation of
the surface Gaussian curvature. So we see that, through the small
gradient approximation, \emph{a gradient flow pursuing the
minimization of the mean curvature leads to a evolution which
favors the growth of the Gaussian curvature}. The term
proportional to $K_2$ comes from the minimization of the squared
mean curvature. A functional involving the squared mean curvature
is known as Willmore functional and it has its own status within
differential geometry~\cite{willmore}. The bilaplacian
accompanying $K_2$ is the corresponding linearized Euler-Lagrange
equation of the Willmore functional when looking for flat
minimizers, and it has already appeared in the context of
mathematical elasticity~\cite{hornung}. Finally the term
proportional to $K_3$ comes from the minimization of the cubic
power of the mean curvature and it involves a nonlinear
combination of Laplacians of the field. We note that from a more
puristic geometrical viewpoint one would retain only even powers
of the mean curvature in expansion~(\ref{expansion}), which would
give rise to a symmetric solution to the corresponding
simplification of equation~(\ref{parabolic}) (i. e., a solution
invariant to the transformation $u \to -u$). However, from a
physical viewpoint, we are seeking for a solution to a partial
differential equation which represents the interface between two
different media (solid structure and vacuum in the present case)
so this symmetry is not guaranteed a priori, and we need to retain
the odd powers of the mean curvature in
expansion~(\ref{expansion}).

For our current purposes we will focus on the associated
stationary problem to a simplification of
equation~(\ref{parabolic}). Such an equation can be obtained
employing well known facts from the theory of non-equilibrium
surface growth. We may invoke classical scaling arguments in the
physical literature to disregard the last term as a higher order
correction which will not be present in the description of the
largest scale properties of the evolving surface~\cite{barabasi}.
This practically reduces to setting $K_3=0$ in
equation~(\ref{parabolic}). In epitaxial growth one may
phenomenologically set $K_0=0$, and we will assume so for the rest
of this work. The underlying physical reason is that the diffusion
proportional to $K_0$ is triggered by the effect of gravity on
adatoms, and this effect is negligible in the case of epitaxial
growth~\cite{barabasi}. The resulting equation reads
\begin{equation}
\label{parabolic2} u_t = 2 \, K_1 \, \det \left( D^2 u \right) -
K_2 \, \Delta^2 u + \xi(x,t).
\end{equation}
This partial differential equation can be thought of as been an
analogue of equation~(\ref{ssg}). Indeed, it has been shown that
this equation might constitute a suitable description of epitaxial
growth in the same sense equation (\ref{ssg}) is so, and it even
shows more intuitive geometric properties~\cite{escudero}. So, at
the physical level, we can consider equation~(\ref{parabolic2}) as
a higher order conservative counterpart of the Kardar-Parisi-Zhang
equation. At the mathematical level we can consider it as a sort
of Gaussian curvature flow~\cite{chow,andrews} which is stabilized
by means of a higher order viscosity term. Furthermore, this
viscosity term, as we have seen, has a clear geometrical meaning.
As we explain above, in this work we are concerned with the
stationary version of (\ref{parabolic2}), which reads
\begin{equation}\label{Pro0}
\left\{\begin{array}{rcl} \Delta^2 u&=&\text{det} \left( D^2 u
\right) +\lambda f, \qquad x\in \Omega\subset\mathbb{R}^2, \\
\text{boundary}&\,& \text{ conditions,}
\end{array}
\right.
\end{equation}
after getting rid of the equation constant parameters by means of
a trivial re-scaling of field and coordinates. Our last assumption
is that the forcing term $f=f(x)$ is time independent. This type
of forcing is known in the physical literature as columnar
disorder, and it has an actual experimental meaning within the
context of non-equilibrium statistical mechanics~\cite{hh}. The
constant $\lambda$ is a measure of the intensity of the rate at
which new particles are deposited, and for physical reasons we
assume $\lambda \ge 0$ and $f(x) \ge 0$. We will devote our
efforts to rigorously and numerically clarify the existence and
multiplicity of solutions to this elliptic problem when set on a
radially symmetric domain.

\section{Radial problems}
\label{radialproblems}

\subsection{Dirichlet boundary conditions}

We start looking for radially symmetric solutions of boundary
value problem~(\ref{Pro0}) with $f = f(r)$, where $r$ is the
radial coordinate, and homogeneous Dirichlet boundary conditions.
We set the problem on the unit disk. That is, we look for
solutions of the form $u=\tilde{u}(r)$ where
$$r = \sqrt{x_1^2+x_2^2}.$$
By means of a direct substitution we find
\begin{equation}
\label{fullradial} \frac{1}{r} \left\{ r \left[ \frac{1}{r} \left(
r \tilde{u}' \right)' \right]' \right\}' = \frac{1}{r} \,
\tilde{u}' \tilde{u}'' + \lambda f(r),
\end{equation}
where $'=\frac{d}{dr}$, and the conditions $\tilde{u}'(0)=0$,
$\tilde{u}(1)=0$, $\tilde{u}'(1)=0$, and $\lim_{r \to 0} r
u'''(r)=0$; the first one imposes the existence of an extremum at
the origin and the second and third ones are the actual boundary
conditions. The fourth boundary condition is technical and imposes
higher regularity at the origin. If this condition were removed
this would open the possibility of constructing functions $u(r)$
whose second derivative had a peak at the origin. This would in
turn imply the presence of a measure at the origin when
calculating the fourth derivative of such an $u(r)$, so this type
of function cannot be considered as an acceptable solution
of~(\ref{fullradial}) whenever $f(r)$ is a function. Throughout
this section we will assume $f \in L^1([0,1], r \, dr)$, that is
$f$ is an absolutely integrable function against measure $r \, dr$
on the unit interval, and we drop the tilde on $\tilde{u}$ in
order to simplify the notation.

Now we proceed to prove the existence of at least two solutions to
this boundary value problem. From now on we will employ the functional
space $\mathring{W}^{2,2}([0,1], r \, dr)$, which is the closure of
the space of radially symmetric smooth functions compactly supported
inside the unit ball of $\mathbb{R}^2$ with the norm of
$W^{2,2}([0,1], r \, dr)$. We will look for solutions to our problem
within this functional space.

\begin{lemm}
Differential equation~(\ref{fullradial}) subjected to Dirichlet
boundary conditions is the Euler-Lagrange equation of functional
\begin{equation}
\label{radialfunc} \left\{\begin{array}{rcl} J_\lambda:
&\mathring{W}^{2,2}([0,1], r \, dr)&\rightarrow \mathbb{R}
\\& u& \rightarrow J_\lambda \left(u \right) = \frac{1}{2}
\int_0^1 \left[ \left(u'' \right)^2 + \frac{\left(u'
\right)^2}{r^2} \right] r \, dr \\ & & \hspace{2 cm} + \frac{1}{6}
\int_0^1 \left(u' \right)^3 dr - \lambda \int_0^1 f \, u \, r \,
dr.
\end{array}
\right.
\end{equation}
\end{lemm}

\begin{proof}
We consider Euler first variation of functional~(\ref{radialfunc})
\begin{eqnarray}
\left. \frac{d}{dt} J_\lambda(u+t \phi) \right|_{t=0} &=& \\ &=&
\int_0^1 \left[ u'' \phi'' + \frac{u' \phi'}{r^2} \right] r \, dr
+ \frac{1}{2} \int_0^1 \left( u' \right)^2 \phi' dr -\lambda
\int_0^1 f \, \phi \, r \, dr \nonumber \\ \nonumber &=& \int_0^1
\left( \frac{1}{r} \left\{ r \left[ \frac{1}{r} \left( r \, u'
\right)' \right]' \right\}' - \frac{1}{r} \, u'\, u'' - \lambda f
\right) \phi \, r \, dr,
\end{eqnarray}
where the last equality is obtained by means of integration by
parts and application of the boundary conditions, and $\phi$
belongs to $\mathring{W}^{2,2}([0,1],r \, dr)$ but it is otherwise
arbitrary.
\end{proof}

The existence and multiplicity of solutions to our boundary value
problem will be obtained by searching critical points of
functional~(\ref{radialfunc}). We start proving a result
concerning the geometry of this functional.

\begin{lemm}
Functional~(\ref{radialfunc}) admits the following radial (in the
Sobolev space) lower bound:
\begin{equation}
J_\lambda (u) \ge g(||u''||_{L^2(\mu)}) \qquad \mathrm{{\it
where}} \qquad g(x)=\frac{1}{2} \, x^2 - C_1 \, x^3 - C_2 \,
\lambda \, ||f||_{L^1(\mu)} \, x,
\end{equation}
$C_1, C_2>0$ and $\mu$ stands for the radial two-dimensional
measure.
\end{lemm}

\begin{proof}
We have the following chain of inequalities
\begin{eqnarray}
\nonumber J_\lambda \left( u \right) \ge \frac{1}{2} \int_0^1
\left( u'' \right)^2 r \, dr + \frac{1}{2} \int_0^1
\left( u' \right)^2 dr + \frac{1}{6} \int_0^1 \left( u' \right)^3 dr - \lambda \, ||f||_{L^1(\mu)} \, ||u||_{L^\infty(\mu)} \ge \\ \nonumber \\
\nonumber \frac{1}{2} \int_0^1 \left( u'' \right)^2 r \, dr +
\frac{1}{2} \int_0^1 \left( u' \right)^2 dr + \frac{1}{6}
\int_0^1 \left( u' \right)^3 dr - C \, \lambda \, ||f||_{L^1(\mu)} \left[ \int_0^1 \left( u' \right)^2 dr \right]^{1/2} \ge \\
\nonumber \\ \nonumber \frac{1}{2} \int_0^1 \left( u'' \right)^2 r
\, dr - C_1 \left[ \int_0^1 \left( u'' \right)^2 r \, dr
\right]^{3/2} - C_2 \, \lambda \, ||f||_{L^1(\mu)} \left[ \int_0^1
\left( u'' \right)^2 r \, dr \right]^{1/2} = \\ \nonumber
\\ \frac{1}{2} \, || u''||_{L^2(\mu)}^2 - C_1 \,
|| u''||_{L^2(\mu)}^3 - C_2 \, \lambda \, ||f||_{L^1(\mu)} \, ||
u''||_{L^2(\mu)},
\end{eqnarray}
where we have used that $r \in [0,1]$ together with H\"older
inequality in the first inequality, a one-dimensional Sobolev
embedding together with the fact that $||u||_{L^\infty(\mu)} \le
||u||_{L^\infty([0,1])}$ in the second inequality, while in the
third inequality we have disregarded a non-negative quantity, we
have employed two-dimensional Sobolev embeddings and the auxiliary
inequalities
\begin{eqnarray}\nonumber
\int_0^1 \left( u' \right)^2 dr &\le& \left( \int_0^1 (u')^6 \, r \, dr  \right)^{1/3} \left( \int_0^1 \frac{dr}{\sqrt{r}} \right)^{2/3}, \\
\int_0^1 \left( u' \right)^3 dr &\le& \left( \int_0^1 |u'|^9 \, r
\, dr  \right)^{1/3} \left( \int_0^1 \frac{dr}{\sqrt{r}}
\right)^{2/3}, \nonumber
\end{eqnarray}
resulting from the application of H\"older inequality.
\end{proof}

It is clear that for $0 < \lambda < \lambda_c$ small enough the
function $g(x)$ has a negative local minimum and a positive local
maximum. It is also clear that there exist $\varphi, \psi \in
\mathring{W}^{2,2}([0,1], r \, dr)$ such that the following
properties are fulfilled:
\begin{itemize}

\item[] $$ \mathrm{a)} \qquad \int_0^1 f \, \varphi \, r \, dr >
0,
$$

\item[] $$ \mathrm{b)} \qquad \int_0^1 (\psi')^3 \, dr < 0. $$

\end{itemize}
Therefore we find $J_\lambda(s \, \varphi) <0$ for $s$ small
enough and $J_\lambda(s \, \psi) <0$ for $s$ large enough.
Consequently the geometric requirements of the {\it mountain pass}
theorem are fulfilled~\cite{ambrosetti}. Now we move to prove the
compactness requirements. We start verifying a local Palais-Smale
condition for our functional $J_\lambda$.

\begin{Defi}
We say $\{u_n\}_{n \in \mathbb{N}} \subset
\mathring{W}^{2,2}([0,1], r \, dr)$ is a Palais-Smale sequence for
$J_\lambda$ at the level $L$ if the following two properties are
fulfilled:
\begin{itemize}

\item[] $$ \mathrm{1)} \qquad J_\lambda(u_n) \to L \qquad
\mathrm{{\it when}} \qquad n \to \infty ,$$

\item[] $$ \mathrm{2)} \qquad J'_\lambda(u_n) \to 0 \qquad
\mathrm{{\it in}} \qquad \{\mathring{W}^{2,2}([0,1], r \, dr)\}^*
.$$

\end{itemize}
\end{Defi}

Now we prove the following compactness result for $J_\lambda$:

\begin{prop}\label{compactness}
Every bounded Palais-Smale sequence for $J_\lambda$ at the level
$L$ admits a strongly convergent subsequence in
$\mathring{W}^{2,2}([0,1], r \, dr)$.
\end{prop}

\begin{proof}
Since $\{u_n\}_{n \in \mathbb{N}} \subset
\mathring{W}^{2,2}([0,1], r \, dr)$ is bounded we find that, up to
passing to a subsequence, the following properties hold:
\begin{itemize}

\item[I.-] $u_n \rightharpoonup u$ weakly in
$\mathring{W}^{2,2}([0,1], r \, dr)$,

\item[II.-] $u'_n \to u' $ strongly in $L^p ([0,1], r \, dr)$ for
every $1 \le p < \infty$,

\item[III.-] $u_n \to u$ uniformly in $[0,1]$.

\end{itemize}
We write the convergence condition $J'_\lambda(u_n) \to 0$ in
$\{\mathring{W}^{2,2}([0,1], r \, dr)\}^*$ in the following
fashion
\begin{eqnarray}\nonumber
\frac{1}{r} \left\{ r \left[ \frac{1}{r} \left( r u_n' \right)'
\right]' \right\}' &=& \frac{1}{r} \, u_n' u_n'' + \lambda f +
w_n,
\\ \nonumber \\ \nonumber u_n \in \mathring{W}^{2,2}([0,1], r \, dr), \qquad
&w_n \to 0& \qquad \mathrm{in} \qquad \{\mathring{W}^{2,2}([0,1],
r \, dr)\}^*,
\end{eqnarray}
where the $w_n$'s are the error terms. Now we multiply this
equation by $u_n-u$ and integrate over the unit interval with the
appropriate measure to get
\begin{eqnarray}\nonumber
\int_0^1 \left[ u_n'' (u_n'' -u'') + \frac{u_n' (u_n' - u')}{r^2}
\right] r \, dr = \\ = \int_0^1 u_n' u_n'' (u_n - u) \, dr +
\lambda \int_0^1 f (u_n - u) r \, dr + \langle w_n, u_n -u
\rangle, \label{pairing}
\end{eqnarray}
after integration by parts on the first line. The three summands
on the second line converge to zero in the limit $n \to \infty$ by
the above listed properties I. (the third summand) and III. (the
first and second summands). On the other hand we have
\begin{equation}\label{misslapl}
\int_0^1 \left[ (u_n'' - u'') \frac{u_n'}{r} - (u_n'' - u'')
\left( u'' + \frac{u'}{r} \right) - \frac{u_n' - u'}{r} \right(
u'' + \frac{u'}{r} \left) +\frac{u_n' -u'}{r} u_n'' \right] r \,dr
\to 0
\end{equation}
as $n \to \infty$ due to convergence property I. and the facts
$$
\int_0^1 u_n'' u_n' \, dr = \frac{1}{2} \left. (u_n')^2
\right|_0^1 =0, \qquad \int_0^1 u'' u' \, dr = \frac{1}{2} \left.
(u')^2 \right|_0^1 =0,
$$
due to the boundary conditions. Now if we sum
expression~(\ref{misslapl}) to the first line of~(\ref{pairing})
we obtain
\begin{equation}
\int_0^1 |\Delta (u_n -u)|^2 \, r \, dr \to 0 \qquad \mathrm{as}
\qquad n \to \infty,
\end{equation}
where $\Delta = \partial_{rr} + r^{-1} \partial_r$ is the radial
Laplacian, and thus the desired conclusion.
\end{proof}

Before moving to the main result of this section we need one last
technical lemma. We introduce the cutoff function $\Upsilon$ which
is assumed to be non-increasing, smooth and given by
$\Upsilon(t)=1$ if $t \le \ell$ and $\Upsilon(t)=0$ if $t \ge
\ell^\ast$ for two given real numbers $\ell^\ast > \ell >0$.

\begin{lemm}\label{lemmj0}
The functional defined as
\begin{equation}
J^0_\lambda (u)= \frac{1}{2} \int_0^1 \left[ \left(u'' \right)^2 +
\frac{\left(u' \right)^2}{r^2} \right] r \, dr + \frac{1}{6}
\int_0^1 \left(u' \right)^3 \, \Upsilon \left( ||\Delta u ||_2
\right) dr - \lambda \int_0^1 f \, u \, r \, dr,
\end{equation}
fulfills the following properties for suitable values of $\ell$,
$\ell^\ast$ and $\lambda$:
\begin{itemize}

\item[i.-] If $||\Delta u ||_2 < \ell$ then $J_\lambda^0 =
J_\lambda$.

\item[ii.-] If $J_\lambda^0 <0$ then $||\Delta u ||_2 < \ell$.

\item[iii.-] If $\mathfrak{m} = \inf_{u \in
\mathring{W}^{2,2}([0,1],r \, dr)} J_\lambda^0 (u)$ then
$J_\lambda$ verifies a local Palais-Smale condition at the level
$\mathfrak{m}$.

\end{itemize}
\end{lemm}

\begin{proof}
Property i. is obvious. For $\lambda_c >0$ small enough the lower
radial bound $g$ of $J_\lambda$ attains a maximum at a positive
level of ``energy'' for $0<\lambda<\lambda_c$. We denote as $x_0$
the smaller root of $g(x)$ and as $x_m$ the location of the
maximum. Now we choose $\ell =x_0$ and $\ell^\ast =x_m$.
Functional $J^0_\lambda$ admits the following radial lower
bound
$$
J^0_\lambda (u) \ge h(||u''||_{L^2(\mu)}) \qquad \mathrm{where}
\qquad h(x)=\frac{1}{2} \, x^2 - C_1 \, x^3 \, \Upsilon(x) - C_2
\, \lambda \, ||f||_{L^1(\mu)} \, x,
$$
where $C_1$ and $C_2$ are the same constants as in Lemma 3.2.
So this functional is bounded from below and positive for $x
> x_0$. Thus property ii. is fulfilled.

Property iii. follows from the fact that all Palais-Smale
sequences of minimizers of this functional are bounded since
$\mathfrak{m} <0$ together with an application of
Proposition~\ref{compactness}.
\end{proof}

Now we state the main result of this section:
\begin{theo}
There exists a positive real number $\lambda_c$ such that for $0 <
\lambda < \lambda_c$ Dirichlet problem~(\ref{fullradial}) has at
least two solutions.
\end{theo}

\begin{proof}
The functional $J_\lambda$ is well defined in
$\mathring{W}^{2,2}([0,1], r \, dr)$ as the Sobolev inequalities
immediately reveal. One of the key points of our proof is the application of
Ekeland's version of the mountain pass theorem. Our functional
fulfills the regularity required to this end, that is, continuity,
Gateaux differentiability and weak$-*$ continuity of its
derivative. We will prove the existence of two solutions to our
boundary value problem by finding two critical points of
functional $J_\lambda$, one of them is a negative local minimum
and the other one is a positive mountain pass critical point.

We start proving the existence of the local minimum at a negative
level of ``energy''. Our proof will be based on the arguments
in~\cite{gp} for solving problems with concave-convex semilinear
nonlinearities. For $\lambda_c >0$ small enough the lower radial
estimate $g$ attains a maximum at a positive level of ``energy''
for $0<\lambda<\lambda_c$. In the proof of
Lemma~\ref{lemmj0} we have shown that functional $J_\lambda^0$ is
bounded from below and positive for $x > x_0$. Accordingly,
$\mathfrak{m}$ is a negative critical value of $J^0_\lambda$, and
thus of $J_\lambda$, from where we conclude the existence of a
local minimum.

Next we move to prove the existence of a positive mountain pass
critical point. We have already proved the existence of a negative
local minimum, which will be denoted as $u^{(0)}$ from now on. We
know $J_\lambda(u^{(0)})<0$ and we know there exists $u^{(2)}$
with $\left|\left|[u^{(2)}]''\right|\right|_{L^2(\mu)}$ large
enough such that $J_\lambda(u^{(2)})<J_\lambda(u^{(0)})$. We
introduce the set of paths in the Banach space
$$
\Theta=\left\{ \left. \theta \in \mathcal{C}\left( [0,1],
\mathring{W}^{2,2}([0,1], r \, dr) \right) \right| \,
\theta(0)=u^{(0)}, \, \theta(1)=u^{(2)} \right\}.
$$
We introduce as well the value
$$
\wp = \inf_{\theta \in \Theta} \max_{s \in [0,1]}
J_\lambda[\theta(s)],
$$
and apply Ekeland's variational principle~\cite{ekeland} to prove
the existence of a Palais-Smale sequence at it. This means there
exists a sequence $\{u_n\}_{n \in \mathbb{N}} \subset
\mathring{W}^{2,2}([0,1], r \, dr)$ such that $J_\lambda(u_n) \to
\wp$ as $n \to \infty$ and $J_\lambda'(u_n) \to 0$ in $\{
\mathring{W}^{2,2}([0,1], r \, dr) \}^*$.

We must now prove that this Palais-Smale sequence is bounded. For
$u \in \mathring{W}^{2,2}([0,1], r \, dr)$ the following equality
holds
$$
-\int_0^1 u' \, u'' \, u \, dr = - \left. \frac{1}{2} (u')^2 \, u
\right|_0^1 + \frac{1}{2} \int_0^1 (u')^3 \, dr = \frac{1}{2}
\int_0^1 (u')^3 \, dr.
$$
We select $\{u_n\}_{n \in \mathbb{N}}\subset
\mathring{W}^{2,2}([0,1], r \, dr)$ Palais-Smale sequence for
$J_\lambda$ at level $\wp$ and denote $\langle z_n, u_n \rangle =
\langle J_\lambda ' (u_n),u_n \rangle$ to find
$$
\wp + o(1)= J_\lambda(u_n) -\frac{1}{3} \langle J_\lambda '
(u_n),u_n \rangle +\frac{1}{3} \langle z_n, u_n \rangle \ge
$$
$$
\frac{1}{6} \int_0^1 \left[ (u_n'')^2 + \frac{(u_n')^2}{r^2}
\right] r \, dr - \frac{2}{3} C_2 \, \lambda \, ||f||_{L^1(\mu)}
\, ||u_n''||_{L^2(\mu)} + \frac{1}{3} \langle z_n, u_n \rangle \ge
$$
$$
C \, ||u_n''||_{L^2(\mu)},
$$
for a suitable positive constant $C$, large enough $n$ and small
enough $\lambda$. In consequence the sequence is bounded in
$\mathring{W}^{2,2}([0,1], r \, dr)$.

We know, by Proposition~\ref{compactness}, that $J_\lambda$
satisfy a local Palais-Smale condition at the level $\wp$, so we
have $J_\lambda(u^{(1)})= \lim_{n \to \infty} J_\lambda (u_n)=\wp
>0$. Also, $u^{(1)}$ is a mountain pass critical point, and in consequence
$J_\lambda'(u^{(1)})=0$, so our differential equation is
fulfilled in $\mathring{W}^{2,2}([0,1], r \, dr)$.
\end{proof}

\subsection{Navier boundary conditions}

In this section we consider again problem~(\ref{fullradial}) on the
unit interval but this time subjected to Navier boundary conditions.
In the radial setting these conditions translate to $u(1)=0$ and
$u''(1)+u'(1)=0$, and we also assume the extremum condition
$u'(0)=0$ at the origin for symmetry reasons. We again assume $f
\in L^1([0,1], r \, dr)$.

As in the previous section we prove the existence of at least two
solutions to this boundary value problem. Our functional framework
will be given by the space $\hat{W}^{2,2}([0,1], r \, dr)$, which we
define as the intersection $\mathring{W}^{2,2}([0,1], r \, dr) \cap
\mathring{W}^{1,2}([0,1], r \, dr)$. We will look for solutions to
our problem belonging to this functional space and which fulfill the
boundary condition $u''(1)+u'(1)=0$. Note that, in principle, it is
not clear how this condition is fulfilled, because the second
derivatives are just square integrable. However, if we consider the
linear problem
$$
\Delta^2 u = f,
$$
$$
u=0, \quad \Delta u = 0,
$$
in $\Omega \in \mathbb{R}^2$ open, bounded and provided with a
smooth boundary, we find $u \in W^{3,p}(\Omega) \, \forall \, 1 \le
p < 2$ for $f \in L^1(\Omega)$. Consequently $u \in
W^{3-1/p,p}(\partial \Omega)$ and we can interpret this boundary
condition in the sense of traces.

In this case the solutions to the differential equation correspond
to critical points of a slightly different functional.

\begin{lemm}
Differential equation~(\ref{fullradial}) subjected to Navier
boundary conditions is the Euler-Lagrange equation of functional
\begin{equation}
\label{navierfunc} \left\{\begin{array}{rcl} I_\lambda:
&\hat{W}^{2,2}([0,1], r \, dr)&\rightarrow \mathbb{R}
\\& u& \rightarrow I_\lambda \left( u \right)= \frac{1}{2} \int_0^1 \left( u'' +
\frac{u'}{r} \right)^2 r \, dr + \frac{1}{6} \int_0^1 \left( u'
\right)^3 dr -\lambda \int_0^1 f \, u \, r \, dr.
\end{array}
\right.
\end{equation}
\end{lemm}

\begin{proof}
We consider Euler first variation of functional~(\ref{navierfunc})
\begin{eqnarray}
\left. \frac{d}{dt} I_\lambda(u+t \phi) \right|_{t=0} &=& \\ &=&
\int_0^1 \left( u'' + \frac{u'}{r} \right) \left( \phi'' +
\frac{\phi'}{r} \right) r \, dr + \frac{1}{2} \int_0^1 \left( u'
\right)^2 \phi' dr -\lambda \int_0^1 f \, \phi \, r \, dr
\nonumber
\\ \nonumber &=& \int_0^1 \left( \frac{1}{r} \left\{ r \left[
\frac{1}{r} \left( r \, u' \right)' \right]' \right\}' -
\frac{1}{r} \, u'\, u'' - \lambda f \right) \phi \, r \, dr,
\end{eqnarray}
where the last equality is obtained by means of integration by
parts and application of the boundary conditions, and $\phi$
belongs to $\hat{W}^{2,2}([0,1],r \, dr)$ but it is otherwise
arbitrary.
\end{proof}

Now we prove a result concerning the geometry of $I_\lambda$.
First we note that both $J_\lambda$ and $I_\lambda$ are well
defined in $W^{2,2}([0,1],r \, dr)$, the space of all functions
$u:[0,1] \longrightarrow \mathbb{R}$ whose second derivative
($u''$) and first derivative normalized by the independent
variable ($u'/r$) are square integrable on the unit interval
against measure $r \, dr$, as can be seen by means of a direct
application of the Sobolev inequalities.

\begin{lemm}
Let $u \in W^{2,2}([0,1],r \, dr)$. Then $I_\lambda(u) \ge
J_\lambda(u)$.
\end{lemm}

\begin{proof}
We want to prove
\begin{equation}
I_\lambda \left( u \right) \ge \frac{1}{2} \int_0^1 \left[ \left(
u'' \right)^2 + \frac{\left(u'\right)^2}{r^2} \right] r \, dr +
\frac{1}{6} \int_0^1 \left( u' \right)^3 dr -\lambda \int_0^1 f \,
u \, r \, dr.
\end{equation}
This follows from
\begin{equation}
\int_0^1 \left( u'' + \frac{u'}{r} \right)^2 r \, dr = \int_0^1
\left[ \left( u'' \right)^2 + \frac{\left( u'\right)^2}{r^2} + 2
u'' \, \frac{u'}{r} \right] r \, dr,
\end{equation}
and
\begin{equation}
\int_0^1 \left( u'' \, \frac{u'}{r} \right) r \, dr = \int_0^1 u''
\, u' \, dr = \left. \frac{1}{2} \left( u' \right)^2 \right|_0^1
\ge 0,
\end{equation}
because $u'(0)=0$.
\end{proof}

\begin{obs}
Note that this result implies that
the geometry of $I_\lambda$ corresponds to the same mountain pass
shape of $J_\lambda$.
\end{obs}

In the following we will prove the existence of at least two
solutions in this case too. The proofs run in parallel to those of
the previous section, so we will simply adapt the arguments and
write exclusively those parts in which the differences are
explicit.

\begin{prop}\label{propnavier}
Every bounded Palais-Smale sequence for $I_\lambda$ at the level
$L$ admits a strongly convergent subsequence in
$\hat{W}^{2,2}([0,1], r \, dr)$.
\end{prop}

\begin{proof}
Since $\{u_n\}_{n \in \mathbb{N}} \subset \hat{W}^{2,2}([0,1], r
\, dr)$ is bounded we find that, up to passing to a subsequence,
the following properties hold:
\begin{itemize}

\item[I.-] $u_n \rightharpoonup u$ weakly in $\hat{W}^{2,2}([0,1],
r \, dr)$,

\item[II.-] $u'_n \to u' $ strongly in $L^p ([0,1], r \, dr)$ for
every $1 \le p < \infty$,

\item[III.-] $u_n \to u$ uniformly in $[0,1]$.

\end{itemize}
We write the convergence condition $I'_\lambda(u_n) \to 0$ in
$\{\hat{W}^{2,2}([0,1], r \, dr)\}^*$ in the following fashion
\begin{eqnarray}\nonumber
\frac{1}{r} \left\{ r \left[ \frac{1}{r} \left( r u_n' \right)'
\right]' \right\}' &=& \frac{1}{r} \, u_n' u_n'' + \lambda f +
w_n,
\\ \nonumber \\ \nonumber u_n \in \hat{W}^{2,2}([0,1], r \, dr), \qquad
&w_n \to 0& \qquad \mathrm{in} \qquad \{\hat{W}^{2,2}([0,1], r \,
dr)\}^*,
\end{eqnarray}
where the $w_n$'s are the error terms. Now we multiply this
equation by $u_n-u$ and integrate over the unit interval with the
appropriate measure to get
\begin{eqnarray}\nonumber
\int_0^1 \left\{ \left( u_n'' + \frac{u_n'}{r} \right) \left[ (u_n
-u)'' + \frac{(u_n - u)'}{r} \right] \right\} r \, dr =
\\ \label{convnav1} = \int_0^1 u_n' u_n'' (u_n - u) \, dr + \lambda \int_0^1 f (u_n
- u) r \, dr + \langle w_n, u_n -u \rangle,
\end{eqnarray}
after integration by parts on the first line. The three summands
on the second line converge to zero in the limit $n \to \infty$ by
the above listed properties I. (the third summand) and III. (the
first and second summands). On the other hand we have
\begin{equation}\label{convnav2}
\int_0^1 \left\{ \left( u'' + \frac{u'}{r} \right) \left[ (u_n
-u)'' + \frac{(u_n - u)'}{r} \right] \right\} r \,dr \to 0
\end{equation}
as $n \to \infty$ due to convergence property I.

Now if we subtract expression~(\ref{convnav2}) from the first line
of~(\ref{convnav1}) we obtain
\begin{equation}
\int_0^1 |\Delta (u_n -u)|^2 \, r \, dr \to 0 \qquad \mathrm{as}
\qquad n \to \infty,
\end{equation}
where $\Delta = \partial_{rr} + r^{-1} \partial_r$ is the radial
Laplacian, and thus the desired conclusion.
\end{proof}

\begin{theo}
There exist a positive real number $\lambda_c$ such that for $0 <
\lambda < \lambda_c$ the Navier problem for~(\ref{fullradial}) has
at least two solutions.
\end{theo}

\begin{proof}
The functional $I_\lambda$ is well defined in
$\hat{W}^{2,2}([0,1], r \, dr)$ as the Sobolev inequalities
immediately reveal. As in the previous section, we will prove the
existence of two solutions to our boundary value problem by
finding two critical points of functional $I_\lambda$, one of them
is a negative local minimum and the other one is a positive
mountain pass critical point. The proof of existence of the
minimum is identical in both cases, so it will not be reproduced
herein.

So we concentrate in proving the existence of the positive
mountain pass critical point. We employ the same minimax technique
as in the previous section and the existence of a Palais-Smale
sequence $\{u_n\}_{n \in \mathbb{N}} \subset \hat{W}^{2,2}([0,1],
r \, dr)$ such that $J_\lambda(u_n) \to \wp$ and $J_\lambda'(u_n)
\to 0$ as $n \to \infty$ in $\{ \hat{W}^{2,2}([0,1], r \, dr)
\}^*$, where $\wp$ is the critical mountain pass level.

We must now prove that this Palais-Smale sequence is bounded. For
$u \in \hat{W}^{2,2}([0,1], r \, dr)$ the following equality holds
$$
-\int_0^1 u' \, u'' \, u \, dr = - \left. \frac{1}{2} (u')^2 \, u
\right|_0^1 + \frac{1}{2} \int_0^1 (u')^3 \, dr = \frac{1}{2}
\int_0^1 (u')^3 \, dr.
$$
We select $\{u_n\}_{n \in \mathbb{N}}\subset \hat{W}^{2,2}([0,1],
r \, dr)$ Palais-Smale sequence for $I_\lambda$ at level $\wp$ and
denote $\langle z_n, u_n \rangle = \langle I_\lambda ' (u_n),u_n
\rangle$ to find
$$
\wp + o(1)= I_\lambda(u_n) -\frac{1}{3} \langle I_\lambda '
(u_n),u_n \rangle +\frac{1}{3} \langle z_n, u_n \rangle \ge
$$
$$
\frac{1}{6} \int_0^1 \left( u_n'' + \frac{u_n'}{r} \right)^2 r \,
dr - \frac{2}{3} C_2 \, \lambda \, ||f||_{L^1(\mu)} \,
||u_n''||_{L^2(\mu)} + \frac{1}{3} \langle z_n, u_n \rangle \ge
$$
$$
C \, ||u_n''||_{L^2(\mu)},
$$
for a suitable positive constant $C$, large enough $n$ and small
enough $\lambda$. In consequence the sequence is bounded in
$\hat{W}^{2,2}([0,1], r \, dr)$.

We know, by Proposition~\ref{propnavier}, that $I_\lambda$
satisfies a local Palais-Smale condition at the level $\wp$, so we
have $I_\lambda(u_*)= \lim_{n \to \infty} I_\lambda (u_n)=\wp >0$.
Also $u_*$ is a mountain pass critical point, so $I_\lambda'(u_*)=0$
and our differential equation is
fulfilled in $\hat{W}^{2,2}([0,1], r \, dr)$.
\end{proof}

\section{Numerical results}

So far we have proven the existence of at least two solutions to
both Dirichlet and Navier problems. In this section we will
clarify the nature of these solutions by means of numerically
solving the boundary value problems employing a shooting method.
Our first step will be transforming differential
equation~\eqref{fullradial} into a form more suitable for the
numerical treatment. To this end and from now on we will assume
$f(r) \equiv 1$.

Integrating once equation \eqref{fullradial} against measure $r \,
dr$ and using boundary condition $\lim_{r \to 0} r u'''(r)=0$
yields
\begin{equation}
r \left[ \frac{1}{r} \left( r \tilde{u}' \right)' \right]' =
\frac{1}{2} (\tilde{u}')^2 + \frac{1}{2} \lambda r^2.
\end{equation}
By changing variables $w=r u'$ we find the equation
\begin{equation}\label{numerics}
w'' -\frac{1}{r} \, w' = \frac{1}{2} \, \frac{w^2}{r^2} +
\frac{1}{2} \, \lambda \, r^2.
\end{equation}

We have performed some numerical simulations with the final value
problem for this ordinary differential equation using a
fourth-order Runge-Kutta method. We have employed the final
conditions $w(1)=0$ and $w'(1)$ arbitrary, which correspond to
Dirichlet boundary conditions, to check how big $\lambda$ could be
in order to have solutions. We have solved this problem for $r \in
[0,1]$ and we have looked for solutions such that $\lim_{\epsilon
\to 0^{+}} \, w(\epsilon)/\epsilon=0$, which corresponds to the
extremum condition $u'=0$ for the original differential equation.
The results of the simulations are represented in
figure~\ref{shooting1}. One observes that for $\lambda=0$ there
are one trivial and one non-trivial solutions. For $0 < \lambda <
\lambda_c$ there are two non-trivial solutions which approach each
other for increasing $\lambda$. In particular, the smaller of
these solutions corresponds to a minimum of the ``energy''
functional and the larger solution corresponds to a mountain pass
critical point. In all the calculated cases the minimum solution
is strictly smaller than the mountain pass solution for all $0 \le
r <1$. For $\lambda > \lambda_c$ no more solutions were
numerically found. The critical value of $\lambda$ was numerically
estimated to be $\lambda_c \approx 169$, and it is achieved when
both critical points merge. These numerical experiments suggest no
solutions exist for large enough $\lambda$.

\begin{figure}
\centering \subfigure[]{
\includegraphics[width=0.45\textwidth]{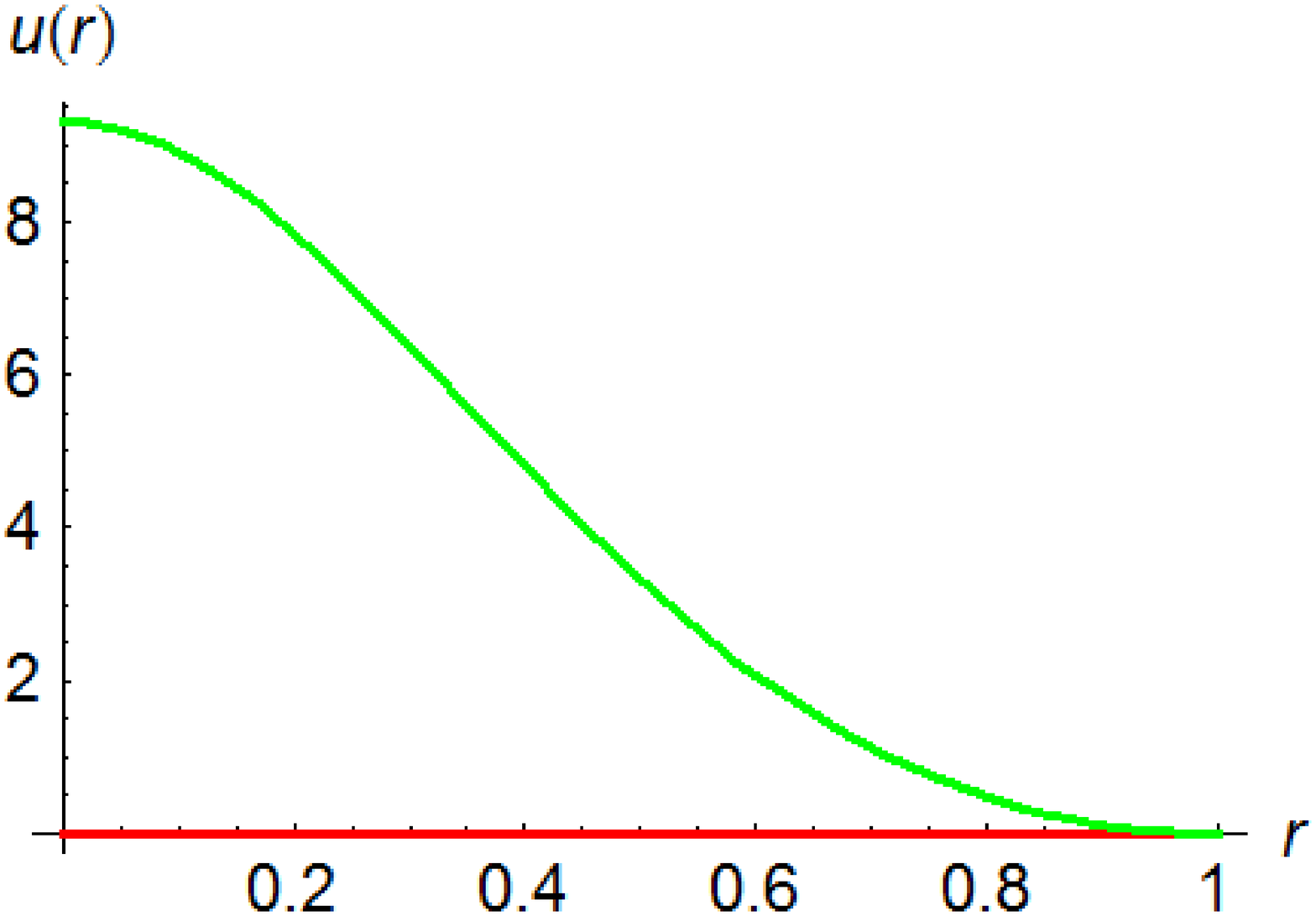}
\label{dirichlet0}} \subfigure[]{
\includegraphics[width=0.45\textwidth]{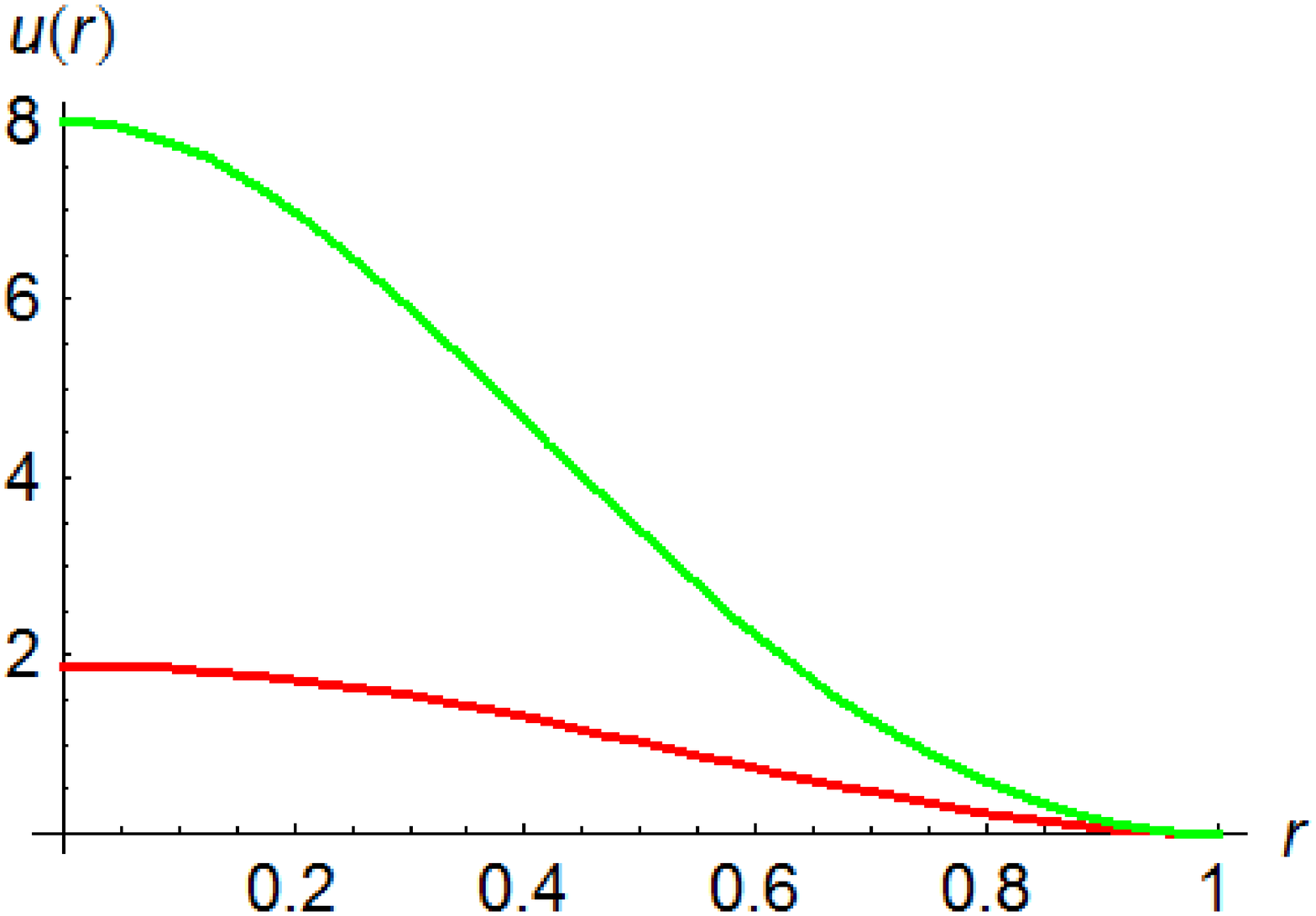}
\label{dirichlet100}}
\\
\subfigure[]{
\includegraphics[width=0.45\textwidth]{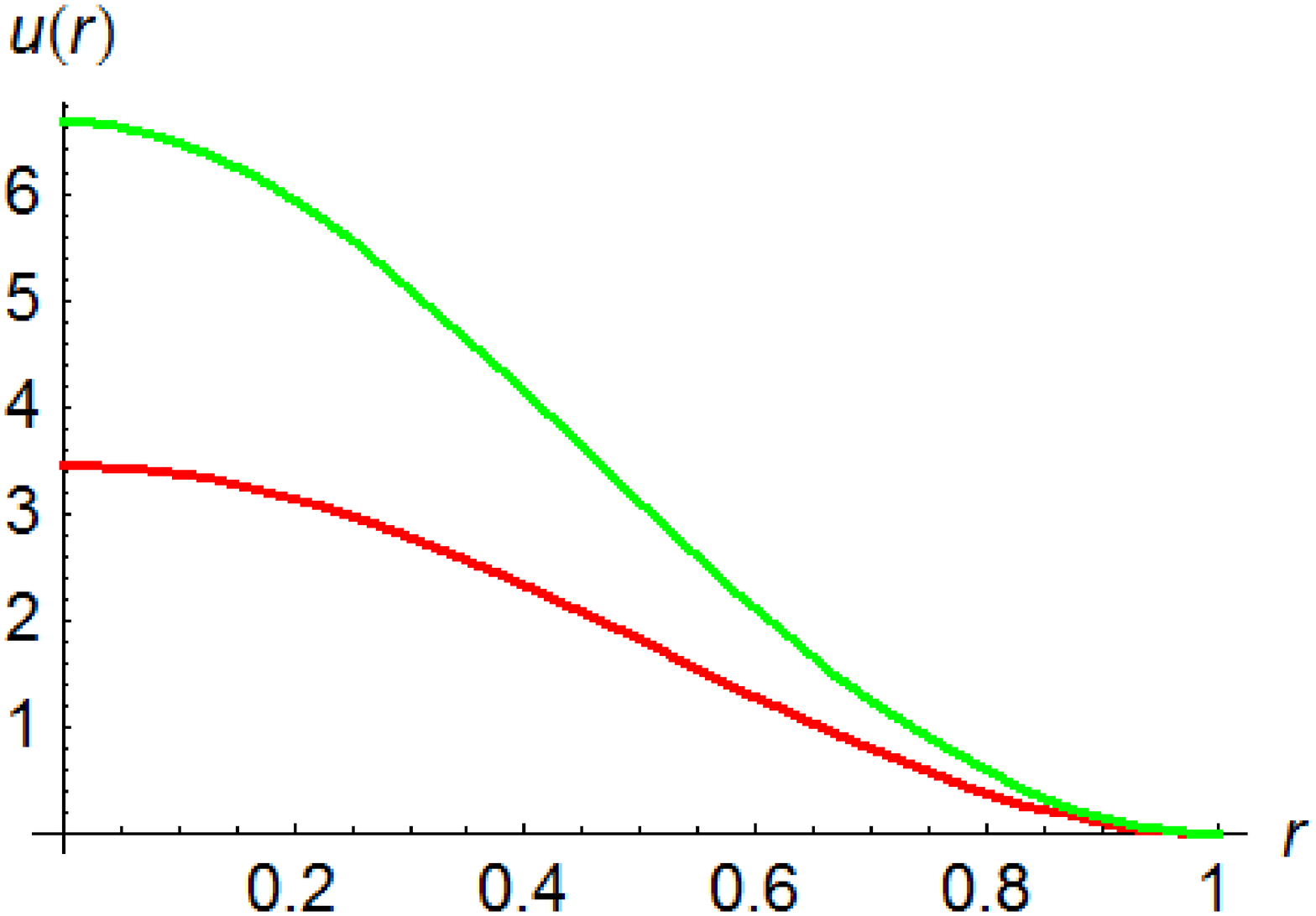}
\label{dirichlet150}} \subfigure[]{
\includegraphics[width=0.45\textwidth]{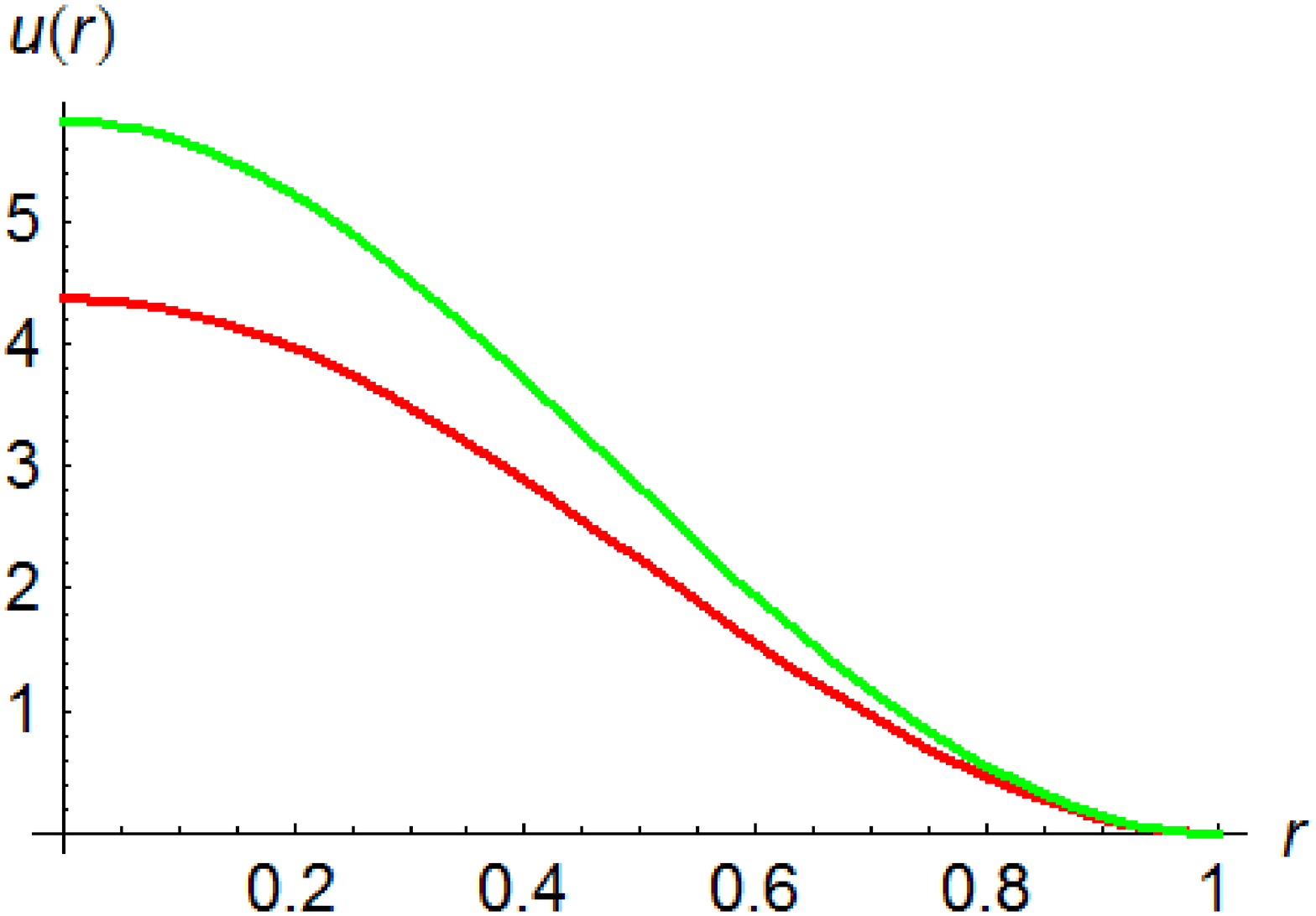}
\label{dirichlet165} } \caption{Radial solutions corresponding to
Dirichlet boundary conditions calculated as explained in the text.
Red line: minimum solution. Green line: mountain pass solution.
Panel \subref{dirichlet0}: $\lambda=0$. Panel
\subref{dirichlet100}: $\lambda=100$. Panel \subref{dirichlet150}:
$\lambda=150$. Panel \subref{dirichlet165}: $\lambda=165$.}
\label{shooting1}
\end{figure}

Now we move back to differential equation~(\ref{fullradial}) but
this time subjected to homogeneous Navier boundary conditions. We
start as above, with the equation
\begin{equation}
w'' -\frac{1}{r} \, w' = \frac{1}{2} \, \frac{w^2}{r^2} +
\frac{1}{2} \, \lambda \, r^2,
\end{equation}
and $w(1)=w'(1)$ arbitrary, what corresponds to homogeneous Navier
boundary conditions.

Also in this case we have employed a  fourth-order Runge-Kutta
method. The results of the numerical experiments are plotted in
figure~\ref{shooting2}. They run in parallel to the results of the
Dirichlet case. We have considered the Navier problem again for $r
\in [0,1]$ and we have searched for solutions such that
$\lim_{\epsilon \to 0^{+}} \, w(\epsilon)/\epsilon=0$, which
corresponds to the extremum condition $u'=0$ for the original
differential equation. Using this shooting method we have found two
different solutions which fulfill these requirements. One observes
that for $\lambda=0$ there are one trivial and one non-trivial
solutions. For $0 < \lambda < \lambda_c$ there are two non-trivial
solutions which approach each other for increasing $\lambda$. For
$\lambda > \lambda_c$ no more solutions were numerically found. The
critical value of $\lambda$ was numerically estimated to be
$\lambda_c \approx 11.34$. Again, the smaller solution corresponds
to a minimum of the ``energy'' functional and the larger solution
corresponds to a mountain pass critical point. In all cases the
minimum solution is strictly smaller than the mountain pass solution
for all $0 \le r <1$.

\begin{figure}
\centering \subfigure[]{
\includegraphics[width=0.45\textwidth]{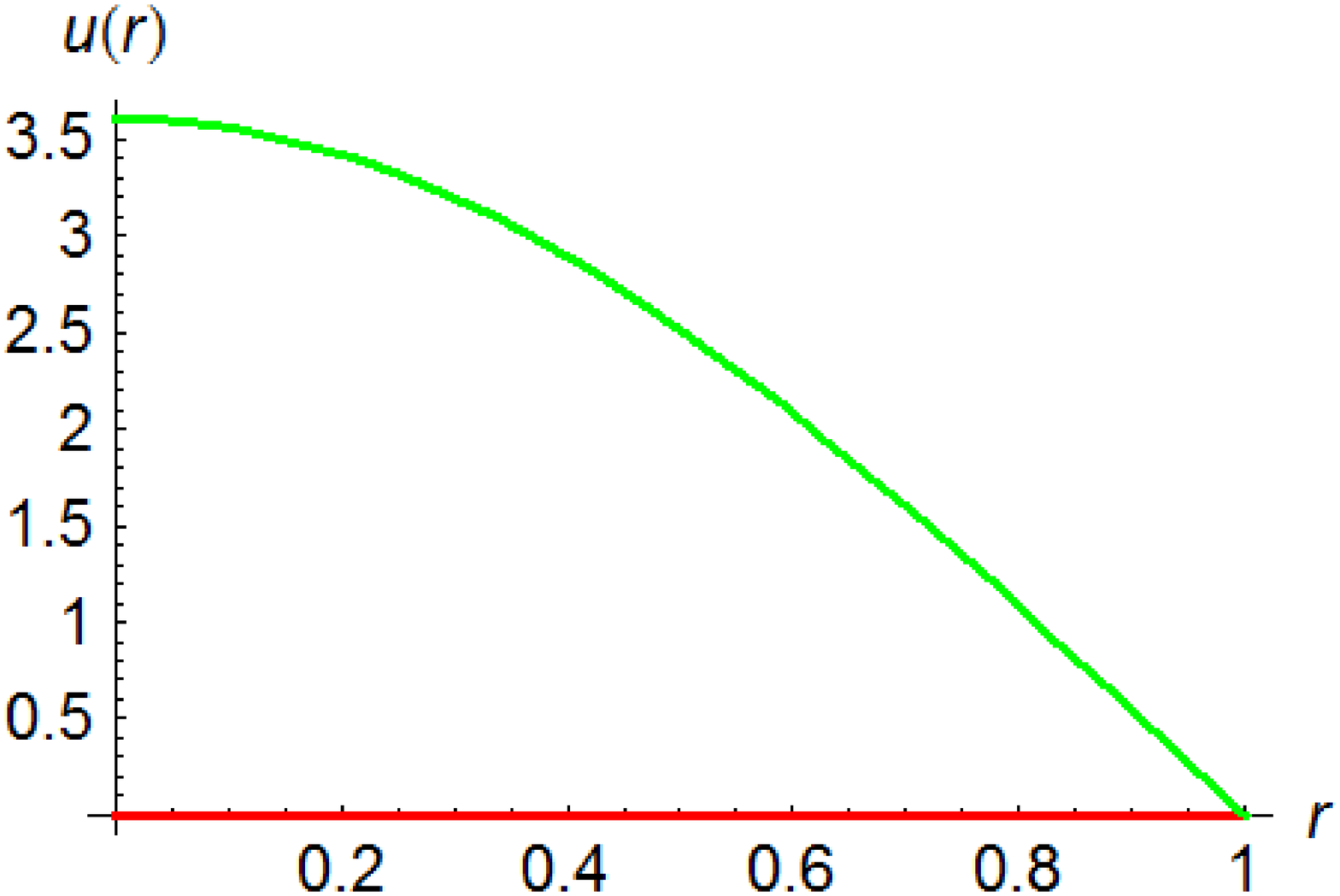}
\label{navier0}} \subfigure[]{
\includegraphics[width=0.45\textwidth]{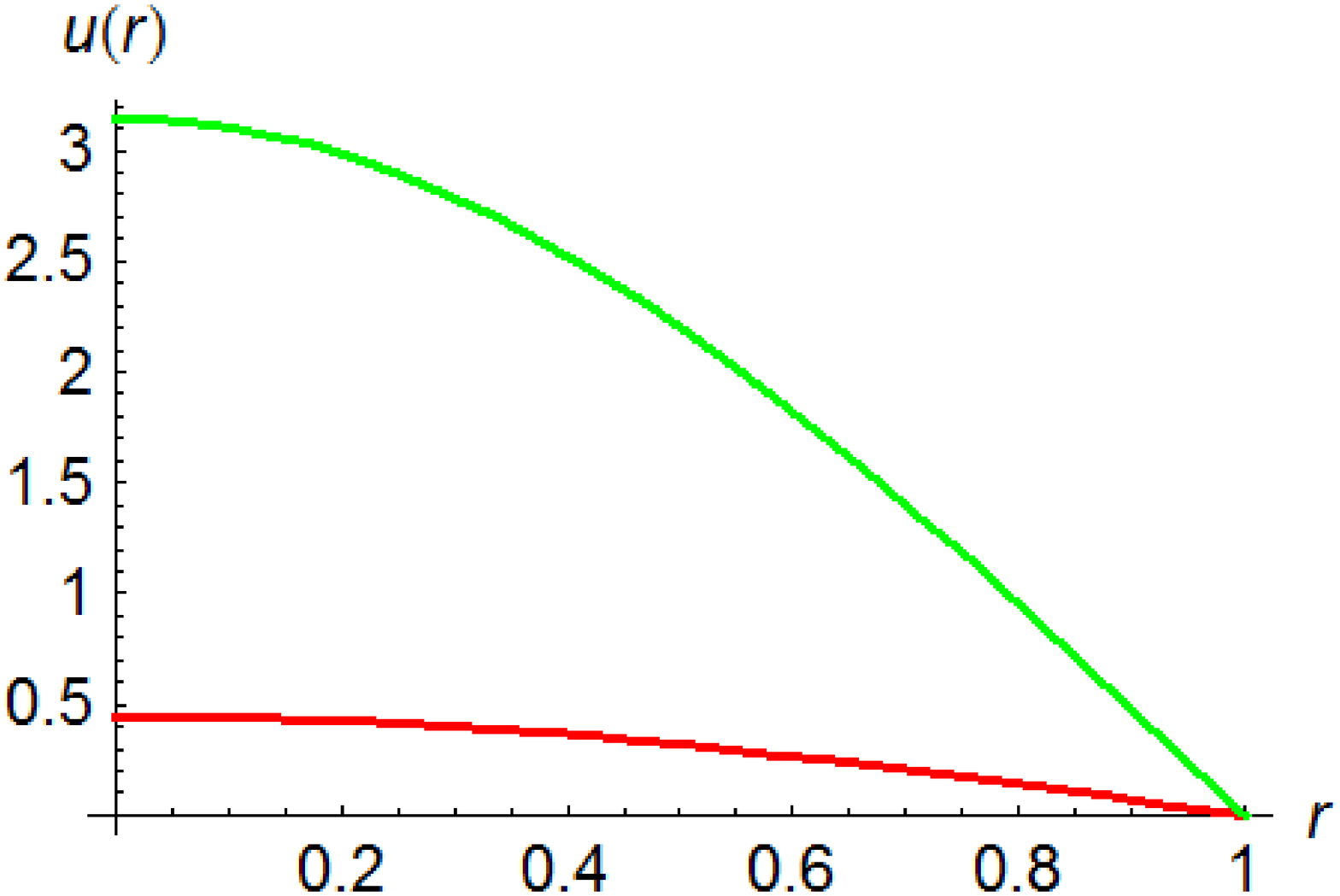}
\label{navier5}} \\ \subfigure[]{
\includegraphics[width=0.45\textwidth]{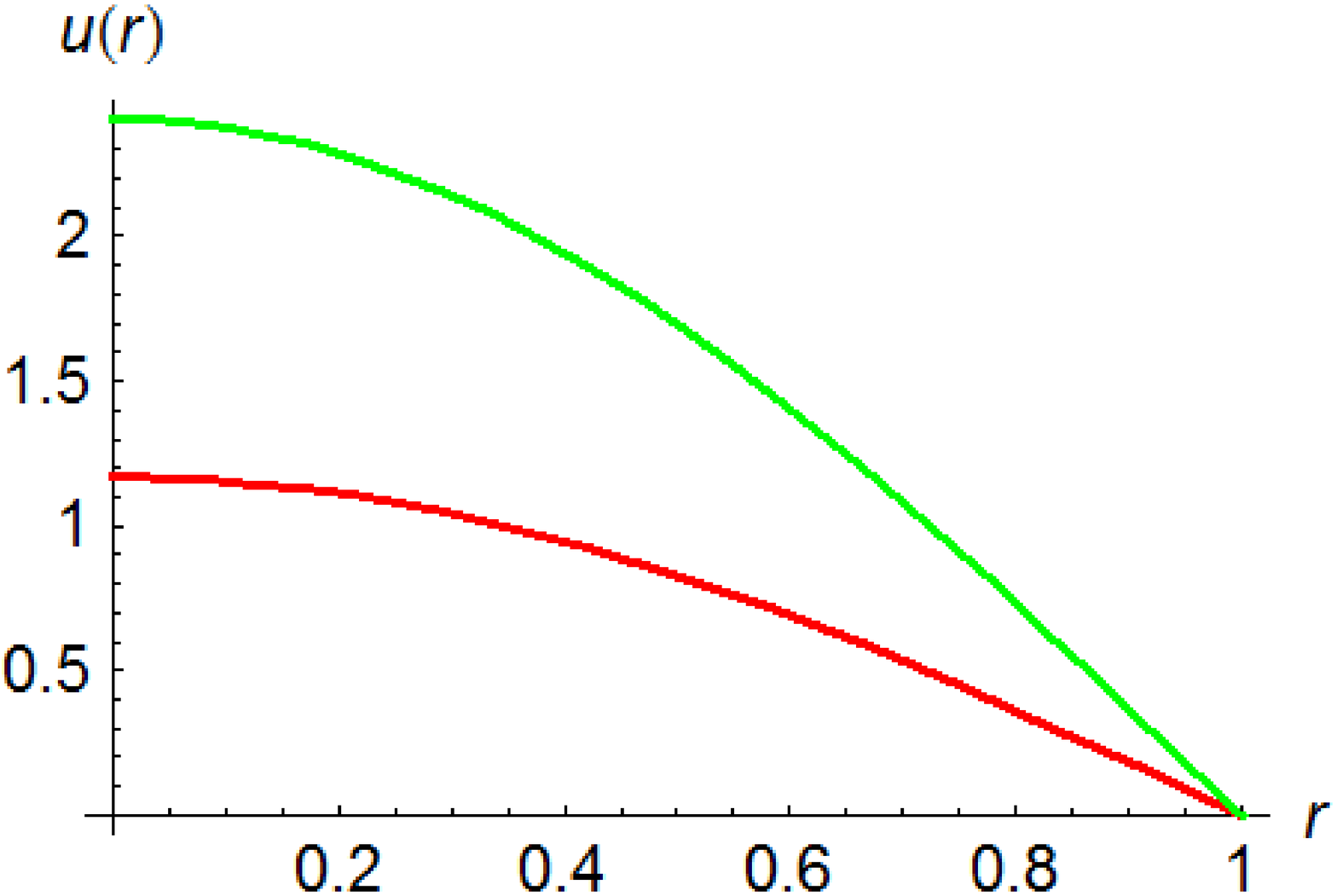}
\label{navier10}} \subfigure[]{
\includegraphics[width=0.45\textwidth]{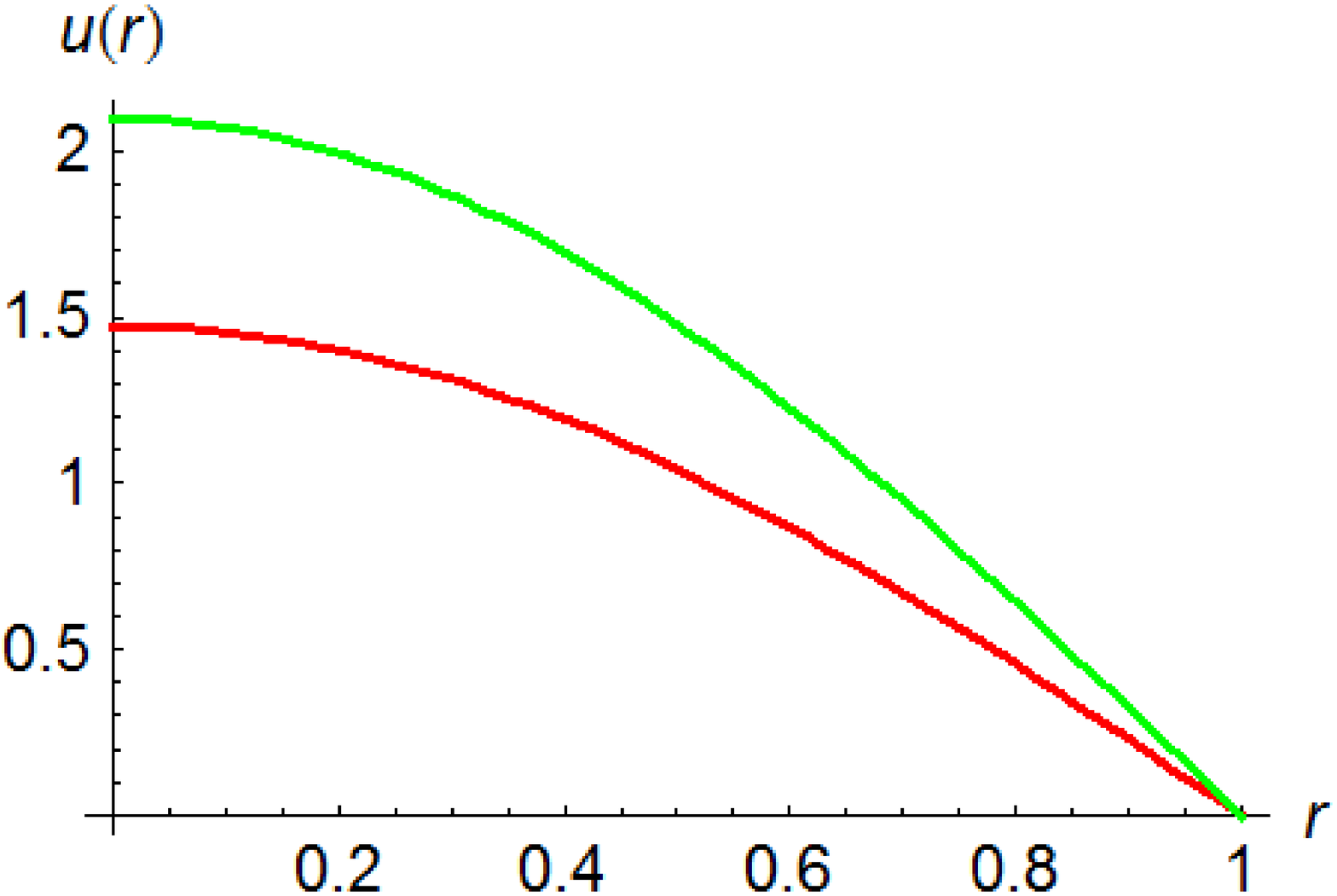}
\label{navier11} } \caption{Radial solutions corresponding to
Navier boundary conditions calculated as explained in the text.
Red line: minimum solution. Green line: mountain pass solution.
Panel \subref{navier0}: $\lambda=0$. Panel \subref{navier5}:
$\lambda=5$. Panel \subref{navier10}: $\lambda=10$. Panel
\subref{navier11}: $\lambda=11$.} \label{shooting2}
\end{figure}

\section{Conclusions and Outlook}

We have analyzed a differential equation appearing in the physical
theory of epitaxial growth. We have started formally introducing
the corresponding partial differential equation and then we have
focused on radial solutions to its stationary counterpart.
The resulting equation has been posed in the unit disk in the plane
subjected to two different sets of boundary conditions. We have
proven the existence of at least two solutions to both boundary
value problems for small enough data.
In each problem we have observed both solutions numerically and
identified one of them with the local minimum of our ``energy''
functional and the other one with a mountain pass critical point.
Due to the qualitatively similar results in both cases, the following
assertions, and in particular the conjectures, refer to both
boundary value problems.
Our numerical simulations have revealed that the solutions are
ordered in the sense that the one corresponding to the minimum
lies strictly below (except for the boundary point $r=1$) the one
corresponding to the mountain pass critical point. We have found
the mountain pass solution is nontrivial for $0 \le \lambda <
\lambda_c$ and the minimum solution is nontrivial for $0 < \lambda
< \lambda_c$ and trivial for $\lambda=0$. We have also proven
nonexistence of solutions for large values of this parameter and we have found
rigorous bounds for the size of the data separating existence from
nonexistence, but the proofs will be reported elsewhere~\cite{preprint}.

We conjecture the solution corresponding to the minimum is
dynamically stable: if we considered the full evolution problem we
would find this solution is locally stable for it. We also
conjecture the mountain pass solution is dynamically unstable. We
have numerically observed both solutions become closer for
$\lambda$ approaching the critical value separating existence from
nonexistence, so we conjecture that the transition from existence
to nonexistence as we vary the parameter $\lambda$ is a
saddle-node bifurcation for the corresponding evolution problem.
We finally conjecture there exists a unique solution, that is
dynamically unstable, for the critical value of $\lambda$,
precisely the one that corresponds to the bifurcation threshold.

On the physical side, our results can be interpreted within the
theory of nonequilibrium potentials~\cite{wio}. The evolution
problems correspond to gradient flows pursuing the minimization of
our ``energy'' functionals, that play the role of nonequilibrium
potentials. If both forcing term and initial condition are small
the system will evolve towards the equilibrium state. If the
forcing were stochastic the equilibrium state would become
metastable. For a large forcing term there are no equilibrium
states, so the system will keep on evolving forever in a genuine
nonequilibrium fashion. In the theory of nonequilibrium growth, in
which the forcing is normally assumed stochastic, it is known that
these features affect both morphology and dynamics of the evolving
interface~\cite{barabasi}. In the case of existence of a local
minimum this would imply in turn the existence of transient
behavior, as found in different models of epitaxial
growth~\cite{vvedensky}. Nonexistence of this state would mean
that the asymptotic state is rapidly achieved. Residence times
could be estimated with the help of the theory of nonequilibrium
potentials~\cite{wio}. Our results constitute a first step towards
the understanding of these phenomena, although more work is needed
in order to get a full understanding of them.

\end{document}